\documentclass[a4paper,11pt]{article}
\usepackage{enumerate}
\usepackage{amsmath,amsthm,amssymb}

\title{Connected Colourings of Complete Graphs and Hypergraphs}
\author{Imre Leader\thanks{Department of Pure Mathematics and Mathematical Statistics, Centre for Mathematical Sciences, University of Cambridge, Wilberforce Road, Cambridge CB3 0WB, United Kingdom. Email: I.Leader@dpmms.cam.ac.uk.} \and Ta Sheng Tan\thanks{Institute of Mathematical Sciences, University of Malaya, 50603 Kuala Lumpur, Malaysia. Email: tstan@um.edu.my. This author acknowledges support received from the University Malaya Research Fund Assistance (BKP) via grant BK021-­2013.} }
\newtheorem{thm}{Theorem}[section]
\newtheorem{prop}[thm]{Proposition}
\newtheorem{lemma}[thm]{Lemma}

\newtheorem{corollary}[thm]{Corollary}
\newtheorem{conjecture}[thm]{Conjecture}

\theoremstyle{remark}

\begin{document}

\maketitle

\begin{abstract}
Gallai's colouring theorem states that if the edges of a complete graph are 3-coloured, with each colour class forming a connected (spanning) subgraph, then there is a triangle that has all 3 colours. What happens for more colours: if we $k$-colour the edges of the complete graph, with each colour class connected, how many of the $\binom{k}{3}$ triples of colours must appear as triangles?

In this note we show that the `obvious' conjecture, namely that there are always at least $\binom{k-1}{2}$ triples, is not correct. We determine the minimum asymptotically. This answers a question of Johnson. We also give some results about the analogous problem for hypergraphs, and we make a conjecture that we believe is the `right' generalisation of Gallai's theorem to hypergraphs.
\end{abstract}

\section{Introduction}

Gallai's colouring theorem (see~\cite{gallai} or~\cite{maffray}) states that if we 3-colour the edges of $K_n$, the complete graph on $n$ vertices, in such a way that each colour class forms a connected spanning subgraph, then there exists a triangle that is \emph{multicoloured}, meaning that no two of its edges have the same colour.
\\\\
What happens if we have 4 colours? Let us call a colouring of $K_n$ \emph{connected} if each colour class forms a connected spanning subgraph. So suppose that we have a connected 4-colouring of $K_n$: of the 4 possible triples of colours, how many must appear as the colour set of a multicoloured triangle? It is easy to see that we must have at least 3 triples. Indeed, if no triangle is coloured as 123 or 124 then, viewing the 4-colouring as a 3-colouring with colours 1, 2 and `3 or 4', we would contradict Gallai's theorem. And it is also immediate that we cannot guarantee all 4 triples (at least if $n$ is large): just take colour classes 1, 2 and 3 to be paths that are `completely unrelated' (i.e., the union of them does not contain a triangle), and let colour class 4 be everything else. This does not have any triangle with colours 123.
\\\\
Johnson~\cite{johnson} asked: what happens if we have more colours? So suppose that we have a connected $k$-colouring of $K_n$. What is the least number of triples that must appear as the colour sets of multicoloured triangles (perhaps for $n$ large)? There is an obvious guess, namely that we repeat the above: so we let $k-1$ of the colour classes be paths, which are completely unrelated, and the other colour class be everything else. This gives $\binom{k-1}{2}$ triples. Is this the right answer?
\\\\
Surprisingly, it turns out that one can do significantly better than this. In Section 2, we give a simple construction to show that the true answer is about $\frac{1}{3}k^2$.
\\\\
In Section 3, we turn our attention to the corresponding question for hypergraphs. We concentrate on the 3-uniform case. Perhaps the first attempt to find an analogue of Gallai's theorem would be to ask: if we 4-colour the set of all 3-sets from an $n$-set, in such a way that each colour class is connected (in some sense or other), must there be a 4-set that is multicoloured (i.e. whose 3-sets receive all 4 colours)? There are several different ways to define `connected', but it turns out, as we will see, that even for the strongest notion of connectedness the answer is that we need not have such a 4-set. However, if we return to 3-colourings, and ask for a 4-set whose 3-sets receive all 3 colours, then we do not know what happens. We make various related conjectures, about this case and the $r$-uniform case.
\\\\
We remark that Gallai's theorem has been the starting point for a considerable amount of work. For example, Ball, Pultr, and Vojt\v{e}chovsk\'{y}~\cite{ball} considered a special class of Gallai graphs, those where each triangle spans precisely two colours, and Gy\'{a}rf\'{a}s, Sark\"{o}zy, Seb\H{o} and Selkow~\cite{gyarfas} considered Ramsey-type results for Gallai colourings. See also~\cite{fujita,gurvich,gyarfas2,gyarfas3} for related results.
\\\\
We write $[k]=\{1,2,\ldots,k\}$. In a $k$-colouring, we usually use colours from $[k]$. We also often refer to `different multicoloured triangles' for multicoloured triangles having different colour sets. 

\section{Multicoloured triangles in coloured complete graphs}

In this section, we consider $f(k)$, the minimum number of triples that can appear as the colour sets of multicoloured triangles in a connected $k$-colouring of $K_n$, for any $n$. (We remark in passing that one might also ask for the minimum provided $n$ is sufficiently large - but in fact, as we will see later in the section, this is the same notion.)
\\\\
We start with an easy lower bound of $f(k)$: any connected $k$-colouring of $K_n$ must contain at least $\frac{k(k-2)}{3}$ different multicoloured triangles. This is a consequence of Gallai's theorem and the following simple lemma.

\begin{lemma} \label{setlemma}
Let $\mathcal{A}$ be a family of subsets of size $3$ of $[k]$ such that whenever we partition $[k]$ into three non-empty subsets, $[k] = R_1 \cup R_2 \cup R_3$, there exists an $A \in \mathcal{A}$ with $A\cap R_i \neq \emptyset$ for $i=1,2,3$. Then $|\mathcal{A}|\geq \frac{k(k-2)}{3}$.
\end{lemma}

\begin{proof}
We show that each element of $[k]$ is in at least $k-2$ sets of $\mathcal{A}$ (whence $|\mathcal{A}|\geq \frac{k(k-2)}{3}$ by double counting). If we fix an element $i\in[k]$ and consider the graph where the edges are induced by the sets containing $i$, then by the condition in the lemma, it is easy to see that this is a connected graph on $k-1$ vertices and so must have at least $k-2$ edges.
\end{proof}

\noindent
For an alternative proof, note that, partitioning $[k]$ into $\{1\} \cup \{2\} \cup \{3,\ldots,k\}$, there must be a set $A_1$ in $\mathcal{A}$ containing $\{1,2\}$ and wlog $A_1 = \{1,2,3\}$. Then partitioning $[k]$ into $\{1\} \cup \{2,3\} \cup \{4,\ldots,k\}$, there must be another set $A_2$ in $\mathcal{A}$ containing $\{1,2\text{ or }3\}$ and wlog $A_2 = \{1,2\text{ or }3,4\}$. Continuing to partition $[k]$ into $\{1\} \cup \{2,3,4\} \cup \{5,\ldots,k\},\{1\} \cup \{2,3,4,5\} \cup \{6,\ldots,k\}, \ldots, \{1\} \cup \{2,\ldots,k-1\} \cup \{k\}$, we can see that there are at least $k-2$ sets in $\mathcal{A}$ containing $1$.

\begin{corollary} \label{lowerbound}
$f(k)\geq \frac{k(k-2)}{3}$.
\end{corollary}

\begin{proof}
 Suppose now that we have a connected $k$-colouring of $K_n$. The subgraph spanned by colours in $R$ is connected for any subset $R$ of $[k]$. If we partition $[k]$ into three non-empty subsets $R_1 \cup R_2 \cup R_3$, Gallai's theorem says that there must exist a multicoloured triangle with colour set intersecting $R_1$, $R_2$ and $R_3$. The family of colour sets of multicoloured triangles now satisfies the condition in Lemma~\ref{setlemma} and hence has size at least $\frac{k(k-2)}{3}$.
\end{proof}

\noindent
We remark that, in the proof of Lemma~\ref{setlemma}, we only considered partitions with a singleton as a class. One might hope to improve this to get a better lower bound on $f(k)$, but the bound in Lemma~\ref{setlemma} is in fact best possible by an inductive construction shown by Diao, Liu, Rautenbach, and Zhao~\cite{diao}. (See the remark after the next result for an explicit construction.)
\\\\
From the above lemma and the paths colouring discussed in the Introduction, we have $\frac{k(k-2)}{3} \leq f(k) \leq \frac{(k-1)(k-2)}{2}$. For the case $k=5$, this gives $f(5)=5$ or $6$, and it is natural to believe that the paths colouring would be the best, suggesting $f(5)=6$. But surprisingly, this is not the case. And in fact this paths colouring is not right in general, not even asymptotically. Indeed, we will give another colouring to improve the upper bound of $f(5)$ and in general $f(k)$.
\\\\
To be able to have a connected $5$-colouring of $K_n$, we need each subgraph to have at least $n-1$ edges, implying that the minimal complete graph to have a connected $5$-colouring is $K_{10}$, with each colour class forming a tree. However, by going up to $K_{11}$, we are able to find a colouring with more symmetry, which turns out to give fewer multicoloured triangles. This is the case $k=5$ of the following result.

\begin{prop} \label{primecolouring}
Let $n=2k+1$ be prime. Then there is a connected $k$-colouring of $K_n$ with precisely $\frac{k(k-2)}{3}$ multicoloured triangles.
\end{prop}

\begin{proof}
Let $V(K_n)=\{0,1,2,\ldots,n-1\}$. We can partition the edge set of $K_n$ into $k$ disjoint spanning cycles $C_i$, $i=1,2,\ldots,k$, where $E(V_i)=\{\{ai,(a+1)i\}:a=0,1,2,\ldots,n-1\}$. Here, we use multiplication and addition mod $n$. We now colour each $C_i$ with a different colour. This colouring is definitely connected as each colour class spans a cycle. It is also not hard to check that each colour is in precisely $k-2$ different multicoloured triangles. Hence the size of the family of colour sets of multicoloured triangles is exactly $\frac{k(k-2)}{3}$.
\end{proof}

\noindent
We remark that for the case when $2k+1$ is prime, the family of colour sets of multicoloured triangles in the above colouring provides an explicit (non-inductive) construction attaining the bound in Lemma~\ref{setlemma}.
\\\\
The colouring in Proposition~\ref{primecolouring} works for $n=2k+1$ - what about colourings for other values of $n$? For a smaller value of $n$, we note that the minimal complete graph to have a connected $k$-colouring is $K_{2k}$. So we can take the coloured $K_{2k+1}$ in the Lemma~\ref{primecolouring} and delete a vertex from it. Very fortunately, each colour class stays connected. For larger values of $n$, the following simple lemma shows that the above colouring is in fact enough to attain the lower bound of $f(k)$, for each $n\geq 2k$.

\begin{lemma} \label{extendlemma}
Suppose that there is a connected $k$-colouring of $K_m$ with $l$ different multicoloured triangles. Then, for any $n\geq m$, there is a connected $k$-colouring of $K_n$ with $l$ different multicoloured triangles.
\end{lemma}

\begin{proof}
Let $c'$ be the above colouring of $K_m$. Partition the vertices of $K_n$ into $m$ non-empty vertex classes, $V_1 \cup V_2 \cup \ldots \cup V_m$. For $u_i \in V_i$ and $v_j \in V_j$, we define a colouring $c$ for $K_n$ as follows. 
\begin{equation*}
c(u_iv_j) = 
\begin{cases}
c'(ij) & \text{ if } i\neq j,\\
c'(12) &  \text{ if } i = j.
\end{cases}
\end{equation*}

\noindent
It is easy to see that $c$ is a connected $k$-colouring of $K_n$ and any multicoloured triangle must have all three vertices from distinct vertex classes. Hence the family of coloured sets of multicoloured triangles of $c$ is exactly the same as the family of colour sets of multicoloured triangles of $c'$.
\end{proof}

\noindent
Combining Proposition~\ref{primecolouring}, Lemma~\ref{extendlemma} and the discussion after Proposition~\ref{primecolouring}, when $2k+1$ is prime we have a connected $k$-colouring of $K_n$ for any $n\geq 2k$ with exactly $\frac{k(k-2)}{3}$ different multicoloured triangles. Together with the lower bound on $f(k)$, this gives the following corollary.

\begin{corollary} \label{exact}
$f(k)=\frac{k(k-2)}{3}$ when $2k+1$ is a prime.\qed
\end{corollary}

\noindent
When $2k+1$ is not prime, we do not know an explicit connected $k$-colouring attaining the lower bound. Instead, we give an inductive colouring where the number of different multicoloured triangles is close to the lower bound in Corollary~\ref{lowerbound}.
\\\\
The following technical lemma states that if a $k$-coloured complete graph satisfies certain conditions, we can extend this colouring to a larger complete graph by adding an extra colour without creating too many new multicoloured triangles. Indeed, only the minimum number (cf. Lemma~\ref{setlemma}) of multicoloured triangles will be created, that is, $k-1$ of them involving this new colour.

\begin{lemma} \label{inductivecolouring}
Let $c$ be a connected $k$-colouring of $K_n$ with the following properties.
\begin{itemize}
 \item There are exactly $l$ different multicoloured triangles.
 \item There are exactly $k-2$ different multicoloured triangles using colour $k$.
 \item The subgraph spanned by colour $k$ is a cycle.
 \item The edges $v_i v_{i+2}$ have the same colour for all $i\in[n]$. (The subscripts are taken mod $n$, so $v_{n+1}=v_1$ and $v_{n+2}=v_2$.)
\end{itemize}
Then, there exists a connected $(k+1)$-colouring $c'$ of $K_{2n}$ with the following properties.
\begin{itemize}
 \item There are exactly $l+k-1$ different multicoloured triangles.
 \item There are exactly $k-1$ different multicoloured triangles using colour $k+1$.
 \item The subgraph spanned by colour $k+1$ is a cycle.
 \item The edges $v_i' v_{i+2}'$ have the same colour for all $i\in[2n]$. (The subscripts are taken mod $2n$, so $v_{2n+1}=v_1$ and $v_{2n+2}=v_2$.)
\end{itemize}
\end{lemma}

\begin{proof}
Suppose $V(K_n)=\{v_1,v_2,\ldots,v_n\}$ and the subgraph spanned by colour $k$ has edges $v_1v_2,v_2v_3,\ldots,v_nv_1$.
\\\\
Let $V(K_{2n})=\{x_1,x_2,\ldots,x_n,y_1,y_2,\ldots,y_n\}$. We define $c'$ on $K_{2n}$ as follows.
\begin{eqnarray*}
c'(x_ix_j) & = & c(v_iv_j),\\
c'(y_iy_j) & = & c(v_iv_j),\\
c'(x_iy_j) & = & 
\begin{cases}
c(v_iv_j) &\text{ if } j\notin\{i,i+1\},\\
k+1 &\text{ otherwise}.
\end{cases}
\end{eqnarray*}
Here we use addition mod $n$, so $x_{n+1}=x_1$ and $y_{n+1}=y_1$.
\\\\
For each $i\in[k]$, the subgraph spanned by colour $i$ in $c'$ is two copies of the subgraph spanned by colour $i$ in $c$ with at least one edge joining them and so connected in $K_{2n}$. The subgraph spanned by colour $k+1$ is just a spanning cycle of $K_{2n}$ and so also connected. Hence, $c'$ is a connected $k+1$-colouring of $K_{2n}$.
\\\\
The number of multicoloured triangles not using colour $k+1$ is exactly $l$. The number of multicoloured triangles using colour $k+1$ but not colour $k$ is the same as the number of multicoloured triangles using colour $k$ in $c$, that is $k-2$. And finally, there is only one multicoloured triangle using both colours $k$ and $k+1$. In total, there are $l+k-2+1=l+k-1$ different multicoloured triangles in $c'$ and the number of different multicoloured triangles using colour $k+1$ is precisely $k-1$, proving the lemma.
\end{proof}

\noindent
From Corollary~\ref{exact}, we know the exact values of $f(k)$ for infinitely many $k$. Applying Lemma~\ref{inductivecolouring} to the explicit colourings in Lemma~\ref{primecolouring}, we have good upper bounds for $f(k)$ for all $k$'s between consecutive primes. Finally, to obtain the limit of $\frac{f(k)}{k^2}$, we need to know the gaps between consecutive primes. It is known (see e.g. \cite{hoheisel, baker}) that there exists a constant $\alpha <1$ such that $p_{n+1}-p_n<p_n^\alpha$ for sufficiently large $n$, where $p_n$ is the $n$th prime. This determines $f(k)$ asymptotically.

\begin{thm}
$f(k) = \frac{k^2}{3}\big(1+o(1)\big)$. \qed
\end{thm}

\noindent
We have shown that $f(k)=\frac{k(k-2)}{3}$ for infinitely many $k$'s, but what is the exact value of $f(k)$ in general? We believe that a colouring attaining the lower bound in Corollary~\ref{lowerbound} always exists, but we have been unable to prove this.

\begin{conjecture}
$f(k)=\Big\lceil \frac{k(k-2)}{3} \Big\rceil$ for all $k\geq 3$. \qed
\end{conjecture}

\section{Multicoloured 4-sets in coloured complete 3-graphs}

In this section, we wish to find analogues of these results for hypergraphs. We will focus on the case of 3-uniform hypergraphs (or 3-graphs for short).
\\\\
An analogue of Gallai's theorem for 3-graphs would be the following statement. Suppose we connectedly (in some sense of connectedness) $4$-colour the edges of the complete 3-graph on $n$ vertices, $K_n^{(3)}$, then must there exist a multicoloured 4-set (that is, a $K_4^{(3)}$ with all its edges having different colours)?
\\\\
The notion of connectedness in hypergraphs can be generalised in a natural way from the connectedness of 2-graphs. If we view connectedness as a `1-set property', then this would just be \emph{pointwise connectedness} (although some authors call this `connectedness', see e.g.~\cite{duchet}), that is to say a 3-graph is pointwise connected when there is a path between every pair of vertices, where a \emph{path} is a sequence of intersecting 3-edges. We say a colouring of $K_n^{(3)}$ is a \emph{pointwise connected colouring}  if the subgraph spanned by each of the colours is pointwise connected on $n$ vertices.
\\\\
It is easy to see that if we take a `cycles' colouring, analogous to the paths colouring from the Introduction, where we take colour classes 1, 2, and 3 to be completely unrelated spanning cycles, and class 4 to be everything else, then this does not contain a multicoloured 4-set. For example, let $n$ be prime and let $V(K_n^{(3)}) = \{0,1,2,\ldots,n-1\}$. We partition the edge set of $K_n^{(3)}$, $E(K_n^{(3)})$ into $\mathcal{A} \cup \mathcal{B} \cup \mathcal{C} \cup \mathcal{D}$, where
\begin{eqnarray*}
\mathcal{A} & = & \{012,123,\ldots,(n-2)(n-1)0,(n-1)01\},\\
\mathcal{B} & = & \{024,246,\ldots,(n-4)(n-2)0,(n-2)02\},\\
\mathcal{C} & = & \{036,369,\ldots,(n-6)(n-3)0,(n-3)03\},\\
\mathcal{D} & = & E(K_n^{(3)}) \setminus (\mathcal{A} \cup \mathcal{B} \cup \mathcal{C}).
\end{eqnarray*}
If we colour the edges in each of these sets differently, then each colour spans a pointwise connected subgraph. It is also easy to check that there is no multicoloured 4-set.
\\\\
Note that the above example can be generalised to a $k$-colouring of the complete 3-graph in the obvious way. This is to say, there is a pointwise connected $k$-colouring of $K_n^{(3)}$ such that it contains no multicoloured 4-set.
\\\\
What if we view connectedness as a 2-set property instead? That is to say, a 3-graph is connected when there is a \emph{strong path}, that is, a path where each of the intersection sizes is precisely two, between every pair of 2-sets. (Note that this is a stronger notion than being a covering, where we say a 3-graph is a \emph{covering} if every 2-set is in some edge. In fact, it is the strongest possible notion of connectness for 3-uniform hypergraphs, apart from topological notions such as spanning a disc.)
Formally, and from now onwards, we say a 3-graph $H$ is \emph{connected} if for any $\{u,v\},\{u',v'\}$ in $V(H)^{(2)}$ there is a strong path $P=\{E_1,E_2,\ldots,E_k\}$ in $H$ such that $\{u,v\}\subset E_1$ and $\{u',v'\}\subset E_k$. And similarly, we say a coloured $K_n^{(3)}$ is \emph{connected} if the subgraph spanned by each of the colours is connected on the $n$ vertices.
\\\\
With this notion of connectedness for 3-graphs, one might hope to have a direct analogue of Gallai's theorem. However, it turns out that the analogous statement is again false. We will first focus on general $k$-colourings, and will comment on the particular case of $k=4$ afterwards.
\\\\
The idea is to inductively blow up a coloured complete 3-graph that contains no multicoloured 4-set and add a new colour to it without creating any multicoloured 4-set.

\begin{thm} \label{norainbow}
Let $k\geq 1$. Then there is a connected $k$-colouring of $K_n^{(3)}$, for some sufficiently large $n$, with no multicoloured 4-set.
\end{thm}

\begin{proof}
The case $k=1$ is trivial. Suppose $c$ is a connected $k$-colouring of $K_n^{(3)}$ with no multicoloured 4-set. We show that we can $(k+1)$-colour $K_{n^2}^{(3)}$ such that it is connected and does not contain any multicoloured 4-set.
\\\\
Let $V\big(K_{n^2}^{(3)}\big) = V_1 \cup V_2 \cup \ldots \cup V_n$, where $V_i=\{v_{ij}:1\leq j \leq n\}$. We define the $(k+1)$-colouring $c'$ as follows.

\begin{equation*}
c'(v_{ix}v_{jy}v_{lz}) = 
\begin{cases}
c(ijl) & \text{ if $i,j,l$ all distinct},\\
c(xyz) & \text{ if $i,j,l$ not all distinct and $x,y,z$ all distinct},\\
k+1 & \text { otherwise}.
\end{cases}
\end{equation*}
We claim that $c'$ is a connected colouring of $K_{n^2}^{(3)}$. We need to check that the subgraph spanned by colour $s \in [k+1]$, $H_s$ is connected. We shall check that for every pair of 2-sets, $\{v_{ix},v_{jy}\},\{v_{pz},v_{qt}\}$, there is always a strong path in $H_s$ between them. We will do the case when $s\in [k]$. The case $s=k+1$ is similar and hence is left for the reader.
\\\\
If all the four vertices are from different blocks or they are all from the same block, it is clear that there is such a path, induced from colouring $c$. Suppose now that they are from three different blocks. There are two cases for this, that is, when $i=j, p\neq q,i\notin \{p,q\}$ and when $i=p,j\neq q,i\notin \{j,q\}$. For the former case, there must be an edge of colour $s$, $E=\{v_{ix},v_{iy},v_{ru}\}$ with $r\notin \{i,p,q\}$ and with the path between $\{v_{ix},v_{ru}\}$ and $\{v_{pz},v_{qt}\}$, induced from colouring $c$, we have the required path. For the latter case, since there is a path of colour $s$ in the colouring $c$ between $\{i,j\}$ and $\{i,q\}$, this induces a path in $H_s$ joining $\{v_{ix},v_{jy}\}$ and $\{v_{iz},v_{qt}\}$. The case when the four vertices are in two different blocks is similar. Hence, $c'$ is indeed a connected colouring.
\\\\
Now, we claim that $c'$ does not span a multicoloured 4-set. Let $\{v_{ix},v_{jy},v_{pz},v_{qt}\}$ be a 4-set. If $i,j,p,q$ or $x,y,z,t$ are all distinct, then the colour of the 4-set is the same as a 4-set induced by $c$ on $K_n^{(3)}$, which is not multicoloured. Suppose now that they are in three different blocks, that is, $i=j, p\neq q,i\notin \{p,q\}$, then $c'(v_{ix}v_{pz}v_{qt})=c'(v_{jy}v_{pz}v_{qt})=c(ipq)$, hence not multicoloured. If they are from two different blocks, there are two cases to consider, that is, when $j=p=q, x=y$ and when $i=j,p=q,x=z$. For the former case, we have $c'(v_{ix}v_{pz}v_{qt})=c'(v_{jy}v_{pz}v_{qt})=c(xzt)$, hence not multicoloured. For the latter case, we have $c'(v_{ix}v_{jy}v_{qt})=c'(v_{jy}v_{pz}v_{qt})=c(xyt)$, also not multicoloured.
\\\\
We have now exhibited a $(k+1)$-colouring of $K_{n^2}^{(3)}$ such that it is connected and contains no multicoloured 4-set. This completes the proof of the theorem.
\end{proof}

\noindent
The theorem above says that we can connectedly 4-colour the complete 3-graph to avoid any multicoloured 4-set In how small a complete 3-graph can this be done? For example, the above colouring requires $n$, the number of vertices, to be about $3^8=6561$.
\\\\
We now show that one may take $n=17$, by giving an explicit connected 4-colouring of $K_{17}^{(3)}$ with no multicoloured 4-set. We suspect that the value of 17 is optimal.

\begin{prop}
There is a connected 4-colouring of $K_{17}^{(3)}$ with no multicoloured 4-set.
\end{prop}

\begin{proof}
We would like to have a very symmetric colouring, and indeed we will have that any two of our colour classes are isomorphic 3-graphs.
Let the vertices of $K_{17}^{(3)}$ be $\{v_0,v_1,\ldots,v_{16}\}$. We define the \emph{distance} of two vertices, $v_i,v_j$ to be $\min\{|i-j|,17-|i-j|\}$. For each edge $v_iv_jv_k$, its `type' is a 3-tuple consisting the three distances of the three pairs of vertices. For example, we say the edge $v_1v_2v_4$ is of type $(1,2,3)$ (or simply type 123 in short).
\\\\
All edges of a given type will receive the same colour. Note that there are 8 special types of edges with a repeated distance, namely $\text{type }112, \text{type }224, \ldots, \text{type } 881$. So each colour class should contain 2 of those and 4 other types of edges.
\\\\
We are now ready to give a 4-colouring without multicoloured 4-set. Let $\mathcal{C}$ be a set of types of edges, namely $\mathcal{C}=\{112,336,145,235,347,458\}$. For a positive integer $k$, we write $k\mathcal{C}=\{k\times C: C\in \mathcal{C}\}$, where $k\times (a,b,c) = (ka\pmod{17}, kb\pmod{17}, kc\pmod{17})$. (Here, we view $x$ as the same as $17-x$.)
\\\\
One can check that $\mathcal{C} \cup 2\mathcal{C} \cup 4\mathcal{C} \cup 8\mathcal{C}$ partitions the types of edge in $K_{17}^{(3)}$. Now we can colour each of the edges of $K_{17}^{(3)}$ by one of four different colours depending on which set its type lies in. 
\\\\
To check this colouring is indeed connected on $K_{17}^{(3)}$, we can check that in the subgraph spannned by each colour, there is a strong path from $\{v_0,v_1\}$ to every other pair of vertices. For example, from $\{v_0,v_1\}$ to $\{v_5,v_9\}$, we have the path $\{v_0v_1v_2, v_0v_2v_5,v_2v_5v_9\}$ in the subgraph spanned by the colour in correspondence to $\mathcal{C}$. Note that we only need to check for the case $\mathcal{C}$, as the four subgraphs spanned by the four colours are isomorphic. The rest of the cases are similar.
\\\\
Suppose now that there is a multicoloured 4-set and one of the edges are from the special types. We may assume that this 4-set is $\{v_0,v_1,v_2,v_x\}$. It is enough to consider the cases when $3\le x\le 9$, and in each of these cases the 4-set is not multicoloured. So a multicoloured 4-set cannot have any special type edge. Suppose now that one of the edges is of type 145; again we may assume that the 4-set is $\{v_0,v_1,v_5,v_x\}$. For each value of $x$, we again claim that the 4-set is not multicoloured. For example, when $x=6$, the edge $v_0v_1v_5$ and the edge $v_1v_5v_6$ have the same colour, and hence not multicoloured. All the remaining cases are similar, and so there is no multicoloured 4-set in this colouring.
\end{proof}

\noindent
From the above, it seems that there is no direct analogue of Gallai's theorem in 3-uniform hypergraphs. But perhaps this is because a multicoloured 4-set is too much to ask for, and maybe we should look for a 3-coloured 4-set instead?
\\\\
In each of the colourings of $K_n^{(3)}$ without any multicoloured 4-set we had, there are many 4-sets that have three different edge colours. We say such 4-sets are \emph{tricoloured}. On the other hand, any non-trivial colouring of $K_n^{(3)}$ using at least two colours contains a 4-set that has at least two different edge colours.
\\\\
So it is natural to ask: given some connectedness condition on the $k$-colouring of $K_n^{(3)}$, must it always contain a tricoloured 4-set? From the colourings we have on $K_n^{(3)}$ that avoid multicoloured 4-sets, one might hope that, for any connectedness condition we apply, such a colouring must contain a tricoloured 4-set.
\\\\
Surprisingly, this is not entirely correct. Indeed, suppose we weaken the condition of connectedness of 3-graphs we had before by only requiring the presence of a path (and not a strong path) between every pair of 2-sets - note that this is exactly the condition of being a covering, as defined earlier. We now give a covering $k$-colouring of $K_n^{(3)}$ (again, this means that every colour class is a covering) without any tricoloured 4-set. This colouring is very similar to the one in Theorem~\ref{norainbow}, but rather easier as we have a weaker notion of connectedness.

\begin{lemma} \label{notricoloured}
Let $k\geq 1$. Then there is a covering $k$-colouring of $K_n^{(3)}$, for some sufficiently large $n$, with no tricoloured 4-set.
\end{lemma}

\begin{proof}
The case $k=1$ is trivial. Suppose $c$ is a covering $k$-colouring of $K_n^{(3)}$ with no tricoloured 4-set. We want to $(k+1)$-colour $K_{n^2}^{(3)}$ such that it is a covering and does not contain any tricoloured 4-set.
\\\\
Let $V\big(K_{n^2}^{(3)}\big) = V_1 \cup V_2 \cup \ldots \cup V_n$, where $V_i=\{v_{ij}:1\leq j \leq n\}$. We define the $(k+1)$-colouring $c'$ as follows.

\begin{equation*}
c'(v_{ix}v_{jy}v_{lz}) = 
\begin{cases}
c(ijl) & \text{ if $i,j,l$ all distinct},\\
c(xyz) & \text{ if $i=j=l$},\\
k+1 & \text { otherwise}.
\end{cases}
\end{equation*}
As in the proof of Theorem~\ref{norainbow}, it is not hard to check that $c'$ is in fact a covering $(k+1)$-colouring of $K_{n^2}^{(3)}$ without any tricoloured 4-set.
\end{proof}

\noindent
Despite the above colouring with no tricoloured 4-set, we still believe that every connected $k$-coloured $K_n^{(3)}$ must contain an tricoloured 4-set. This is our conjectured extension of Gallai's theorem. 

\begin{conjecture}
For all sufficiently large $n$, every connected $3$-colouring of $K_n^{(3)}$ must contain a tricoloured 4-set. \qed
\end{conjecture}

\section{Further remarks and questions}

We remarked after Proposition~\ref{primecolouring} that of course $K_{2r}$ is the minimal complete graph to have a connected $k$-colouring, because a connected 2-graph on $n$ vertices must have at least $n-1$ edges. In order to determine the minimal complete 3-graph having a connected $k$-colouring, we need to know the minimal number of edges of a connected 3-graph on $n$ vertices. We have the following simple result.

\begin{lemma}
Let $H_n$ be a connected 3-graph on $n$ vertices. Then $\big|E(H_n)\big| \geq \big\lfloor \frac{1}{2}\binom{n}{2} \big\rfloor$. Moreover, this bound can be obtained.
\end{lemma}

\begin{proof}
To show the lower bound, we construct a connected 2-graph $G_n$ on $\binom{n}{2}$ vertices from $H_n$. Let the vertex set of $G_n$ indexed by the 2-sets of vetices of $H_n$. For each edge $v_iv_jv_k$ in $H_n$, we add three edges $(v_iv_j)(v_iv_k)$, $(v_iv_j)(v_jv_k)$ and $(v_iv_k)(v_jv_k)$ to $G_n$. By the connectedness of $H_n$, we can see that $G_n$ is connected. In fact, if we delete one of the three edges added to $G_n$ from each edge $v_iv_jv_k$ in $H_n$, $G_n$ remains connected. 
\\\\
By construction, $G_n$ has $2\big|E(H_n)\big|$ edges and together with the fact that $G_n$ being connected implies that it has at least $\binom{n}{2} - 1$ edges, implying $H_n$ must have at least $\big\lfloor \frac{1}{2}\binom{n}{2} \big\rfloor$ edges.
\\\\
For the upper bound, we show by inductive constructions that there is a connected 3-graph on $n$ vertices with $\big\lfloor \frac{1}{2}\binom{n}{2} \big\rfloor$ edges. 
\\\\
We first deal with the case when $n$ is even. Given $H_n$ with $V(H_n)=\{x_1,\ldots,x_k,y_1,\ldots,y_k\}$, we construct $H_{n+4}$ as follows.
\begin{eqnarray*}
V(H_{n+4})&:=&V(H_n) \cup \{a,b,c,d\}, \\
E(H_{n+4})&:=&E(H_n) \cup \{ax_iy_i:1\leq i\leq k\} \cup \{bx_iy_i:1\leq i\leq k-1\} \cup \\
	       & & \{cx_iy_i:1\leq i\leq k\} \cup \{dx_iy_i:1\leq i\leq k\} \cup \{abx_k,abc,acd,bdy_k\}.
\end{eqnarray*}

\noindent
It is not hard to check that $H_{n+4}$ is connected if $H_n$ is connected. We need two base cases, that is, when $n=2,4$. For $n=2$, we can simply take $H_2$ to be the empty 3-graph on two vertices and for $n=4$, we can take $H_4$ to be the complete 3-graph on four vertices taking away an edge. Now $|E(H_{n+4})| = |E(H_n)|+ 2n +3 = \big\lfloor \frac{1}{2}\binom{n}{2} \big\rfloor + 2n+3 = \big\lfloor \frac{1}{2}\binom{n+4}{2} \big\rfloor$.
\\\\
We can now construct a connected 3-graph on $n+1$ vertices from one on $n$ vertices, with $n$ being even. Given $H_n$ with $V(H_n)=\{x_1,\ldots,x_k,y_1,\ldots,y_k\}$, we construct $H_{n+1}$ as follows.
\begin{eqnarray*}
V(H_{n+1})&:=&V(H_n) \cup \{a\}, \\
E(H_{n+1})&:=&E(H_n) \cup \{ax_iy_i\}:1\leq i \leq k \}.
\end{eqnarray*}

\noindent
It is straightforward to check that $H_{n+1}$ is indeed connected and 
$|E(H_{n+1})| = |E(H_n)|+ \frac{n}{2} = \big\lfloor \frac{1}{2}\binom{n}{2} \big\rfloor + \frac{n}{2} = \big\lfloor \frac{1}{2}\binom{n+1}{2} \big\rfloor$.
\end{proof}

\noindent
In Section 3, we tried to extend Gallai's theorem to hypergraphs. Returning to graphs, we could also ask, what about a multicoloured $K_d$ in a connectedly $k$-coloured $K_n$, for any $d>3$? The exact same paths colouring we had in the Introduction shows that there exists a connectedly $k$-coloured $K_n$ without any multicoloured $K_d$. But another question would be, how many colours must some $K_d$ have in a connected $k$-colouring of $K_n$? For example, if we have a connected $6$-colouring of $K_n$, then there must exist a $K_4$ that spans at least four colours - this is a simple consequence of Gallai's theorem plus the fact that every vertex is incident with edges of all colours. In the other direction, we can take five disjoint paths on $n$ vertices such that the union of them contains no cycles of length at most 4 and give the paths colouring (as in the Introduction) to deduce that every $K_4$ spans at most four colours. 

\begin{prop}
Let $3\leq d \leq k$. Then there is a $K_d$ that spans at least $d$ colours in any connectedly $k$-coloured $K_n$. Moreover, for all sufficiently large $n$, there exists a connectedly $k$-coloured $K_n$ with no $K_d$ spanning more than $d$ colours.
\end{prop}

\begin{proof}
As above, the first statement is a simple consequence from Gallai's theorem plus the fact that every vertex is incident with edges of all colours.
\\\\
The latter statement is trivially true for $d=k$. For $d<k$, we can take $k-1$ disjoint paths on $n$ vertices such that the union of them contains no cycles of length at most $d$ and give the paths colouring as the one mentioned in the introduction, that is, colour each of the spanning paths by a different colour and the rest of the edges by another colour, say green. Suppose there is a $K_d$ that spans $d+1$ colours, then there are at least $d$ non-green edges on these $d$ vertices, which implies that there is a cycle of length at most $d$ from the union of these paths, contradicting the assumption.
\end{proof}

\noindent
Until now we have focused on graphs and 3-uniform hypergraphs, but it is natural to seek extensions to the case of general $r$-uniform hypergraphs. As before, we say that an $r$-graph is \emph{connected} if there is a strong path between every pair of $(r-1)$-sets. Here, a strong path is a sequence of $r$-edges where each consecutive pair of $r$-edges has intersection size precisely $r-1$. Again, we say a coloured $K_n^{(r)}$ is \emph{connected} if each colour class spans a connected subgraphs. It appears that the interesting case is still for 3 colours.

\begin{conjecture}
For all sufficiently large $n$, if we connectedly $3$-colour the edges of the complete $r$-graph on $n$ vertices, then there must exist an $(r+1)$-set that uses all three colours. \qed
\end{conjecture}

\noindent
A slightly weaker notion would be to use covering, where we say an $r$-graph is a \emph{covering} if every $(r-1)$-set is in some $r$-edge. We say a colouring of the complete $r$-graph is \emph{covering} if each colour class spans a covering. 
\\\\
Unfortunately, as with 3-graphs (Lemma~\ref{notricoloured}), it is again not true that every weakly connected 3-colouring of a complete 4-graph contains a 5-set that uses all three colours.

\begin{lemma}
For all sufficiently large $n$, there is a covering 3-colouring of $K_n^{(4)}$ with no 5-set that uses all three colours.
\end{lemma}

\begin{proof}
Suppose $c$ is a covering red/blue colouring of $K_n^{(4)}$ and $d$ is a covering blue/green colouring of $K_n^{(4)}$.
\\\\
Let $V\big(K_{n^2}^{(3)}\big) = V_1 \cup V_2 \cup \ldots \cup V_n$, where $V_i=\{v_{ij}:1\leq j \leq n\}$. We can view this as the blow-up of $K_n^{(4)}$ of colouring $d$ with $n$ copies of $K_n^{(4)}$ of colouring $c$. There are three other different types of 4-edges to be coloured. Formally, we define the 3-colouring $c'$ as follows.

\begin{equation*}
c'(v_{ix}v_{jy}v_{pz}v_{qt}) = 
\begin{cases}
d(ijpq) & \text{ if $i,j,p,q$ all distinct},\\
c(xyzt) & \text{ if $i=j=p=q$},\\
red & \text{ if $\big|\{i,j,p,q\}\big|=3$},\\
blue & \text{ if $i=j,p=q,i\ne p$},\\
green & \text { if $i=j=p,q\ne i$}.
\end{cases}
\end{equation*}
It is now straightforward to check that $c'$ is in fact a covering 3-colouring of $K_{n^2}^{(4)}$ without any $K_5^{(4)}$ that uses all three colours.
\end{proof}

\noindent
It seems that the above inductive colouring works because we are lucky to have exactly three colours, namely one to colour each of the three extra types of 4-edges to maintain the connectivity of the blow-up $K_{n^2}^{(4)}$. In fact, we do not see how to generalise this to greater values of $r$, even when we are allowed to use more colours.
\\\\
Finally, returning to Theorem~\ref{norainbow}, it would be interesting to know what happens if the notion of connectedness is strengthened to some topological notion of connectedness (to do with the simplicial complex formed by the triples in each colour class): this is an idea of Thomass\'{e}~\cite{thomasse}.

\end{document}